\documentclass[a4paper,12pt]{amsart}

\usepackage{amsmath,amsfonts,amsthm,amssymb}
\usepackage{fullpage}
\usepackage[svgnames]{xcolor}
\usepackage{hyperref}
\usepackage[all]{xy}
\usepackage{mathrsfs}

\theoremstyle{plain}    

\newtheorem{theorem}[subsection]{Theorem}
\newtheorem{lemma}[subsection]{Lemma}
\newtheorem{proposition}[subsection]{Proposition}
\newtheorem{corollary}[subsection]{Corollary}

\theoremstyle{definition}   
\newtheorem{definition}[subsection]{Definition}
\newtheorem{example}[subsection]{Example}
\newtheorem{question}[subsection]{Question}

\theoremstyle{remark}   
\newtheorem{remark}[subsection]{Remark}

\newcommand{\Int}{\mathbb{Z}}   
\newcommand{\sphere}{\mathbb{S}}

\newcommand{\Derived}{\mathsf{D}}   
\newcommand{\Tri}{\mathsf{T}}       
\newcommand{\St}{\mathsf{St}}       
\newcommand{\loc}{\mathrm{Loc}}     
\newcommand{\thick}{\mathrm{Thick}} 
\newcommand{\compact}{\mathsf{C}}   
\newcommand{\bl}{\mathsf{B}}        
\newcommand{\class}{\mathsf{A}}     
\newcommand{\Ker}{\mathrm{Ker}}     
\renewcommand{\Im}{\mathrm{Im}}     

\newcommand{\Path}{\mathrm{Path}}   
\newcommand{\fl}{\mathrm{L}}        
\newcommand{\fib}{\mathrm{F}}       
\newcommand{\End}{\mathrm{End}}     
\newcommand{\op}{\mathrm{op}}       
\newcommand{\Hom}{\mathrm{Hom}}     
\newcommand{\ext}{\text{Ext}}       
\newcommand{\chains}{\mathrm{C}}    
\newcommand{\ee}{\mathcal{E}}       
\newcommand{\spec}{\mathrm{Spec}}   
\newcommand{\supp}{\mathrm{supp}}   

\newcommand{\uu}{\mathcal{U}}       
\newcommand{\vv}{\mathcal{V}}       
\newcommand{\zz}{\mathcal{Z}}       
\newcommand{\pp}{\mathfrak{p}}      
\newcommand{\qq}{\mathfrak{q}}      
\newcommand{\mm}{\mathfrak{m}}      
\renewcommand{\aa}{\mathfrak{a}}    
\newcommand{\Hring}{\mathscr{H}}    

\newcounter{commentcount}[section]

\begin{document}

\title{On the string topology category of compact Lie groups}
\author{Shoham Shamir}
\address{Shoham Shamir}
\email{shoham\_s@yahoo.com}
\date{\today}

\begin{abstract}
This paper examines the string topology category of a manifold, defined by Blumberg, Cohen and Teleman. Since the string topology category is a subcategory of a compactly generated triangulated category, the machinery of stratification, constructed by Benson, Krause and Iyengar, can be applied in order to gain an understanding of the string topology category. It is shown that an appropriate stratification holds when the manifold in question is a simply conneceted compact Lie group. This last result is used to derive some properties of the relevant string topology categories.
\end{abstract}

\maketitle                  


\section{Introduction}

For a closed, oriented manifold $M$ of dimension $m$, the \emph{string topology category $\St_M$} of $M$ is a category enriched over chain-complexes over a ground ring $k$. Its objects are closed, oriented submanifolds of $M$. The chain-complex of morphisms $\Hom_{\St_M}(N_1,N_2)$ between two objects $N_1,N_2 \in \St_M$ is quasi-isomorphic to the chain-complex of the space $\Path_M(N_1,N_2)$, which is the space of paths from $N_1$ to $N_2$ in $M$. Obviously, for an object $N \in \St_M$ the chain-complex of endomorphisms $\Hom_{\St_M}(N,N)$ forms a differential graded algebra (\emph{dga}). Blumberg, Cohen and Teleman defined the string topology category in~\cite{BlumbergCohenTeleman}, where they also posed the following question:

\begin{question}[{\cite{BlumbergCohenTeleman}}]
\label{que: BCT question}
For which connected, closed, oriented submanifolds $N \subset M$ is the Hochschild cohomology of the dga $\Hom_{\St_M}(N,N)$ isomorphic to the homology of the of the free loop space $\fl M$ as algebras?
\[ H_{*+m}(\fl N;k) \cong HH^*\Hom_{\St_M}(N,N)?\]
The algebra structure on the homology of the free loop space is the \emph{loop product} given by Chas and Sullivan~\cite{ChasSullivan}.
\end{question}

In fact, for every object $N \in \St_M$ there is a natural map of graded-commutative rings
\[\varphi_N: H_{*+m}(\fl M;k) \to HH^*(\Hom_{\St_M}(N,N)|k)\]
whose construction is given in Section~\ref{sec: Hochschild cohomology and Dwyer-Greenlees completion}. So it is natural to ask when is this particular map an isomorphism. The following result provides an answer to this question in several cases.

\begin{theorem}
\label{thm: Answer theorem}
Fix a regular commutative ring $k$ and let $M$ be a compact, simply-connected manifold of dimension $m$ satisfying the following conditions:
\begin{enumerate}
\item $H_*(\Omega M;k)$ is a polynomial ring over $k$ on finitely many generators concentrated in even degrees and
\item the natural map $H_{*+m}(\fl M) \to H_*(\Omega M) $ is surjective.
\end{enumerate}
Then for any connected, closed, oriented submanifold $N \subset M$ the map
\[ \varphi_N: H_{*+m}(\fl M;k) \to HH^*(\Hom_{\St_M}(N,N)|k)\]
is an isomorphism.
\end{theorem}

For rational coefficients $k=\mathbb{Q}$, Theorem~\ref{thm: Answer theorem} provides an answer to the question above whenever $M$ is a simply-connected, compact Lie group. For integral coefficients this yields an answer when $M$ is a special unitary group $SU(n)$, for $n>1$. This is shown in Theorem~\ref{the: The two exmaples of compact Lie groups}.

Before presenting the second result we require some notation. For a connected submanifold $N \in \St_M$ we denote by $\fib_N$ the homotopy fibre of the inclusion $N \to M$. Recall that the homology of $\fib_N$ is a module over the homology of the based loop space $\Omega M$. When $H_*\Omega M$ is a Noetherian commutative ring, the support of the $H_*\Omega M$-module $H_* \fib_N$ will be denoted by $\supp_{H_*\Omega M} H_* \fib_N$. The singular chain-complex of a space $X$, with coefficients in $k$, is denoted by $\chains_*(X;k)$.

\begin{theorem}
\label{thm: Second theorem}
Let $M$ be a compact manifold satisfying the conditions of Theorem~\ref{thm: Answer theorem}. Then for any two connected, orientable submanifolds $N_1, N_2 \subset M$ the natural map
\[ \Hom_{\St_M}(N_1,N_2) \otimes_{\Hom_{\St_M}(N_1,N_1)}^\mathbf{L} \chains_*(N_1;k) \to \chains_*(N_2;k)\]
is a quasi-isomorphism if and only if
\[ \supp_{H_*(\Omega M;k)} H_*(\fib_{N_2};k) \subseteq \supp_{H_*(\Omega M;k)} H_*(\fib_{N_1};k).\]
\end{theorem}


The techniques used in obtaining these two results are interesting in their own right, and we describe them next. Let $R = \chains_*(\Omega M, k)$ be the differential graded algebra of singular chains on the based loop space, its derived category will be simply denoted by $\Derived$. The string topology category $\St_M$ is in fact a subcategory of $\Derived$, and gains its enrichment from the enrichment of $\Derived$ over $k$-chain complexes.

In addition, the objects of $\St_M$ enjoy the special property of being \emph{compact objects} in $\Derived$. Recall that an object $X$ in a triangulated category is called \emph{compact} if morphisms out of $X$ commute with coproducts. The full subcategory of compact objects in $\Derived$ will be denoted by $\Derived^\compact$. For a compact object $X \in \Derived$ there is a natural morphism (given in Definition~\ref{def: definition of varphi}):
\[ \varphi_X:HH^*(R|k) \to HH^*(\Hom_\Derived(X,X)|k)\]
By the work of Malm~\cite{Malm}, $HH^*(R|k) \cong H_{*+m}(\fl M;k)$ as graded rings. Question~\ref{que: BCT question} can therefore be translated to:
\begin{question}
For which connected, closed, oriented submanifolds $N \subset M$ is the map $\varphi_{\chains_*(\fib_N;k)}$ an isomorphism?
\end{question}
To understand the answer to this question we must first recall what are thick subcategories.

\begin{definition}
A full, triangulated, subcategory $\Tri'$ of a triangulated category $\Tri$ is \emph{thick} if $\Tri'$ is closed under direct summands. The thick subcategory generated by an object $X\in \Tri$ is the smallest thick subcategory containing $X$, this subcategory is denoted by $\thick_\Tri(X)$.
\end{definition}

Now suppose our compact manifold $M$ is simply connected and let $X$ be a compact object in $\Derived$. Note that $k$ is also an object of $\Derived$, it is in fact a compact object. As we shall see, if $k$ belongs to the thick subcategory generated by $X$ then $\varphi_X$ is an isomorphism (Proposition~\ref{pro: when varphi is an iso}). Hence, classifying the thick subcategories of $\Derived^\compact$ might answer our questions.

Benson, Iyengar and Krause recently provided a general machinery for performing such classification~\cite{BIKstratifying,BIKgroupStratifying} called \emph{stratification}. A full review of this machinery is provided in Section~\ref{sec: Stratification of a derived category}, but we give a shortened version here. The ingredients for stratification are two: a compactly generated triangulated category, $\Derived$ in our case, and a graded-commutative Noetherian ring $\Hring$ which \emph{acts} on $\Derived$. The action is given by natural transformations, i.e. given $s \in \Hring_n$ there is a natural transformation $s:\Sigma^n 1_\Derived \to 1_\Derived$. For example, the Hochschild cohomology ring $HH^*(R|k)$ acts on $\Derived$. We will take $\Hring$ to be a Noetherian subring of $HH^*(R|k)$. As we shall see, the action of $\Hring$ on $\Derived$ implies that for any object $X\in \Derived$ the homology $H_*X$ is naturally a module over $\Hring$.

To get a classification of the thick subcategories of $\Derived^\compact$ we need the action of $\Hring$ on $\Derived$ to satisfy certain conditions (given in Definition~\ref{def: Stratification}). When these are satisfied, we say that $\Hring$ \emph{stratifies} $\Derived$. The resulting classification is given in terms of certain subsets of the prime ideal spectrum of $\Hring$. Recall that a subset $\vv$ of the prime ideal spectrum of $\Hring$ is called \emph{specialization closed} if whenever $\pp \subset \qq$ and $\pp \in \vv$ then also $\qq \in \vv$.

The next lemma is an immediate consequence of the tools and results of~\cite{BIKstratifying}.

\begin{lemma}
\label{lem: simple application of stratification}
Let $M$, $R$ and $\Derived$ be as above and let $\Hring$ be a Noetherian subring of $HH^*(R|k)$. Suppose that $H_*R$ is finitely generated as an $\Hring$-module and that $\Hring$ stratifies $\Derived$. Then there is an inclusion respecting injection~\cite[Theorem 6.1]{BIKstratifying}:
\[ \left\{
     \begin{array}{c}
        \text{Thick}\\
        \text{subcategories of }\Derived^\compact\\
     \end{array}
   \right\}
   \longrightarrow
   \left\{
     \begin{array}{c}
        \text{Specialization closed subsets of}\\
        \text{the prime ideal spectrum of } \Hring \\
     \end{array}
   \right\}
\]
Where a thick subcategory $\Tri\subset \Derived^\compact$ maps to the subset $\bigcup_{X \in \Tri} \supp_\Hring H_*X$ of the prime ideal spectrum of $\Hring$.
\end{lemma}

\begin{example}
First, we record the following fact: under the conditions of Lemma~\ref{lem: simple application of stratification}, if $\Tri$ is the thick subcategory generated by a compact object $Y$, then
\[ \bigcup_{X \in \Tri} \supp_\Hring H_*X = \supp_\Hring H_*Y\]
Now suppose, for example, that $k$ is a field and that $\Hring$ is a polynomial ring $k[x_1,...,x_n]$ whose generators lie in positive even degrees. Suppose also that the conditions of Lemma~\ref{lem: simple application of stratification} hold. Then $k \in \thick_\Derived X$ for some $X\in \Derived^\compact$ if and only if $\supp_\Hring k \subset \supp_\Hring H_*X$. It is easy to see that the support of $k$ contains only the maximal ideal $\mm = (x_1,...,x_n)$. Since the support of any non-zero $\Hring$-module contains $\mm$, we see that $k$ is in $\thick_\Derived X$ for any $X \in \Derived^\compact$. Therefore $\varphi_X$ is an isomorphism for every $X \in \Derived^\compact$.
\end{example}

Hence, we need a result asserting the existence of stratification under certain conditions. To get such a result we prefer to work in the context of stable homotopy theory, instead of the differential graded context we followed until now. This is because the stable context is more general than the differential graded one.

Thus, we let $\sphere$ denote the sphere spectrum and from now on we work with $\sphere$-modules and $\sphere$-algebras in the sense of of~\cite{EKMM}. We remind the reader that the stable homotopy groups of a spectrum $X$ are denoted by $\pi_*X$.

\begin{theorem}
\label{the: Main algebraic theorem on stratification}
Let $k$ be a commutative $\sphere$-algebra and let $R$ be a $k$-algebra (which is implicitly assumed to be q-cofibrant over $k$). Suppose $R$ satisfies the following conditions:
\begin{enumerate}
\item $\pi_*R$ is a Noetherian graded-commutative regular ring of finite Krull dimension concentrated in even degrees,
\item the natural map $HH^*(R|k) \to \pi_{-*}R$ is surjective and
\item $R$ is compact as an $R \otimes_k R^{op}$-module.
\end{enumerate}
Choose a Noetherian subring $\Hring \subset HH^*(R|k)$ which surjects onto $\pi_*R$. Then the action of $\Hring$ stratifies $\Derived(R)$.
\end{theorem}

\begin{remark}
By \emph{graded-commutative regular} we mean that for any prime ideal $\pp \subset \pi_*R$ the unique homogeneous maximal ideal of the graded localization $(\pi_*R)_\pp$ is generated by a regular sequence of homogeneous elements.
\end{remark}

Both Theorem~\ref{thm: Answer theorem} and Theorem~\ref{thm: Second theorem} easily follow from the result above. It should be noted that the techniques used in proving Theorem~\ref{the: Main algebraic theorem on stratification} share some similarities with those of~\cite{ShamirStratify}.

\subsection*{Organization of this paper}
In Section~\ref{sec: The string topology category} we define the context we will work in and recall the definition of the string topology category from~\cite{BlumbergCohenTeleman}. Section~\ref{sec: Localizing subcategories and localization} reviews a couple of notions of localization, which will be used throughout the paper. One of these notions of localization is the Dwyer-Greenlees completion. In Section~\ref{sec: Hochschild cohomology and Dwyer-Greenlees completion} we explain how Hochschild cohomology behaves with respect to the Dwyer-Greenlees completion and define the map $\varphi$.

Section~\ref{sec: Stratification of a derived category} provides the necessary background on the subject of stratification from~\cite{BIKstratifying}. In Section~\ref{sec: Hochschild cohomology and localization} we examine how Hochschild cohomology behaves with respect to the localization methods used in the machinery of stratification. Section~\ref{sec: Left ring objects} constructs the main tools we need in order to prove Theorem~\ref{the: Main algebraic theorem on stratification}. Finally, Section~\ref{sec: Proofs of the main results} provides the proofs for the main results.

\subsection*{Notation and terminology}
Most of the work in this paper is carried out in the derived category of an $\sphere$-algebra $R$, denoted $\Derived(R)$. However, we could just as well have used a differential graded setup where $R'$ is a dga. The work of Shipley~\cite{ShipleySpectraDGA} shows that for every such dga $R'$ there is a corresponding a $\sphere$-algebra $R$ such that the derived categories $\Derived(R')$ and $\Derived(R)$ are equivalent as triangulated categories. Thus, there is no loss in working with $\sphere$-algebras.

With regards to $\sphere$-algebras and their modules we follow the work of Dwyer, Greenlees and Iyengar~\cite{DwyerGreenleesIyengar} in notation and terminology. Thus, if $k$ is a commutative $\sphere$-algebra the symmetric monoidal product of $k$-modules (i.e. the smash product) will be denoted by $\otimes_k$ as in~\cite{DwyerGreenleesIyengar} (instead of $\wedge_k$). When $k$ is understood from the context we will simply use the notation $\otimes$ for $\otimes_k$. Also following in~\cite{DwyerGreenleesIyengar}, we implicitly assume that $\otimes_k$ is derived, i.e. that the appropriate cofibrant replacements have been performed before applying the functor. This assumption applies also to the $\Hom_R$ and $\otimes_R$ functors, where $R$ is a $k$-algebra (alternately one can assume these are functors defined on the derived category of $R$). For an $R$-module $X$ we denote by $\End_R(X)$ the $k$-algebra $\Hom_R(X,X)$.

For notation and terminology pertaining to triangulated categories we follow Benson, Iyengar and Krause~\cite{BIKstratifying}. One exception to this are the \emph{homotopy groups} of an object $X$ of the derived category $\Derived(R)$, defined by
\[ \pi_nX=\hom_{\Derived(R)} (\Sigma^n R,X) = \ext^{-n}_R (R,X)\]
To avoid cumbersome notation we adopt the convention that $\pi^n X = \pi_{-n} X = \ext^n_R(R,X)$.

Given commutative $\sphere$-algebra $k$ and a $k$-algebra $R$ we denote by $R^e$ the $\sphere$-algebra $R \otimes_k R^\op$. Recall that this is implicitly the derived smash product $R\otimes_k^\mathbf{L}R^\op$. The Hochschild cohomology of $R$ (over $k$) is the graded-commutative ring
\[ HH^*(R) = HH^*(R|k) = \ext^*_{R^e}(R,R)\]



\section{The string topology category and related categories}
\label{sec: The string topology category}

In this section we recall the definition of the string topology category from~\cite{BlumbergCohenTeleman} as well as several other categories which are essential for this paper. We start by describing our context in detail.

\subsection*{Derived categories and compact objects}
Fix a commutative $\sphere$-algebra $k$, in this paper $k$ will usually be the Eilenberg-Mac Lane spectrum of some commutative ring. Let $R$ be a $k$-algebra. The category of $R$-modules has a well known model category structure where the weak equivalences are $\pi_*$-isomorphisms. The resulting homotopy category is a triangulated category called the \emph{derived category} of $R$-modules, denoted by $\Derived(R)$.

Recall that an object $X$ in some triangulated category $\Tri$ is called \emph{compact} if, for any set of objects $\{ Y_i\}_{i\in I}$ in $\Tri$, the natural map
\[ \hom_{\Tri}(X, \oplus_i Y_i) \to \oplus_i \hom_{\Tri}(X, \oplus_i Y_i) \]
is an isomorphism. The full subcategory of compact objects in $\Tri$ is denoted $\Tri^\compact$. For example, the $k$-algebra $R$ is compact in $\Derived(R)$. In fact, $R$ is a \emph{compact generator} of $\Derived(R)$, meaning that $\ext_{\Derived (R)}^*(R,X) = 0$ if and only if $X=0$. Or, in other words, for any $X\in \Derived(R)$, $\pi_*X =0$ if and only if $X=0$.

\begin{example}
Here is our main example of such a derived category. Let $M$ be a pointed topological space and let $\Omega M$ be the space of Moore loops on $M$, which is an associative and unital topological monoid. The \emph{chains of $\Omega M$} with coefficients in $k$ is the $k$-algebra
\[ \chains_*(\Omega M; k) = k \otimes_\sphere \Sigma^\infty \Omega M \]
We denote this algebra simply by $\chains_*\Omega M$ whenever $k$ is understood from the context. Note that when $k$ is the Eilenberg-Mac Lane spectrum of a commutative ring $\tilde{k}$ then $\pi_n\chains_*(\Omega M; k) \cong H_n(M;\tilde{k})$. The derived category $\Derived(\chains_*\Omega M)$ and its subcategory of compact objects will be the main categories we work with.

Topological spaces can induce objects of $\Derived(\chains_*\Omega M)$ in the following manner. Suppose that $N$ is a subspace of $M$. Let $\fib_N$ denote the homotopy fibre of the inclusion map $N \to M$. As is well known, we can take $\fib_N$ to be the space of Moore paths $\Path_M(*,N)$ from the base point of $M$ to $N$. This space has an action of $\Omega M$ and therefore $\chains_* \fib_N$ is naturally an object of $\Derived(\chains_* \Omega M)$. In what follows we will often use this connection between subspaces of $M$ and objects of $\Derived(\chains_*\Omega M)$.
\end{example}

\subsection*{The string topology category}
Let $M$ be a pointed closed oriented manifold and let $k$ be a commutative $\sphere$-algebra. Following~\cite{BlumbergCohenTeleman} we define the \emph{string topology category} $\St_M$ to have as objects all closed oriented submanifolds $N \subset M$. The morphisms in $\St_M$ are:
\[ \hom_{\St_M}(N_1,N_2) = \hom_{\chains_*\Omega M}(\chains_* \fib_{N_1},\chains_* \fib_{N_2})\]
Since the category of $\chains_*\Omega M$-modules is enriched over $k$-modules, we can lift this enrichment to the string topology category by defining:
\[ \Hom_{\St_M}(N_1,N_2) = \Hom_{\chains_*\Omega M}(\chains_* \fib_{N_1},\chains_* \fib_{N_2})\]
In this manner we will consider $\St_M$ as a full subcategory of $\Derived(\chains_*\Omega M)$.

\begin{remark}
In~\cite{BlumbergCohenTeleman} $k$ is a field, $\chains_*(\Omega M;k)$ is a dga and $\St_M$ is enriched over chain-complexes. Strictly speaking this is different from the setup presented here, however there is an easy translation between the two contexts. Let $Hk$ be the Eilenberg-Mac Lane spectrum of $k$, then $\chains_*(\Omega M; Hk)$ is an $Hk$-algebra. An mentioned above, Shipley's results~\cite{ShipleySpectraDGA} show that the two derived categories $\Derived(\chains_*(\Omega M;k))$ and $\Derived(\chains_*(\Omega M;Hk))$ are equivalent. Moreover, the machinery developed in~\cite{ShipleySpectraDGA} provides a way of translating between the two enrichments.
\end{remark}

One important fact, noted in~\cite{BlumbergCohenTeleman}, is that every object of $\St_M$ is compact in $\Derived(\chains_*\Omega M)$. We will make repeated use of this fact in the coming sections.


\section{Localizing subcategories and localization}
\label{sec: Localizing subcategories and localization}

This section reviews notions of localization and colocalization and related concepts. Throughout this section $\Derived$ denotes the derived category of an $\sphere$-algebra $R$.

\subsection*{Thick and localizing subcategories}
A \emph{thick subcategory} of $\Derived$ is a full triangulated subcategory closed under retracts. A \emph{localizing subcategory} is a thick subcategory that is also closed under taking coproducts. The thick (resp. localizing) subcategory \emph{generated} by a given class of objects $\class$ in $\Derived$ is the smallest thick (resp. localizing) subcategory containing $\class$, this subcategory is denoted by $\thick_\Derived(\class)$ (resp. $\loc_\Derived(\class)$).

\begin{remark}
We shall also employ the following terminology from~\cite{DwyerGreenleesIyengar}. For $X$ and $Y$ in $\Derived$ we say that $X$ \emph{builds} $Y$ if $Y \in \loc_\Derived (X)$. We say that $X$ \emph{finitely builds} $Y$ if $Y\in \thick_\Derived(X)$.
\end{remark}

Note that since $\Derived$ is the derived category of $R$, the thick subcategory generated by $R$ is $\Derived^\compact$ and the localizing subcategory generated by $R$ is $\Derived$. Moreover, if $\class \subset \Derived^\compact$ is a set of compact objects then, by a result of Neeman~\cite[Lemma 2.2]{NeemanConnectLocalSmash},
\[ \thick_\Derived(\class) = \loc_\Derived (\class) \cap \Derived^\compact\]
In particular, if $X$ and $Y$ are compact objects of $\Derived$ and $X$ builds $Y$ then $X$ finitely builds $Y$.

\subsection{Localization}
A \emph{localization functor} on $\Derived$ is an exact functor $L: \Derived \to \Derived$ together with a natural morphism $\eta:1_\Derived \to L$ such that $L \eta=\eta L$ and $L\eta$ is a natural isomorphism. One important property of a localization functor $L$ is that $\eta$ induces a natural isomorphism
\[ \hom_\Derived(X,LY) \cong \hom_\Derived(LX,LY) \]
Dually there is the concept of a \emph{colocalization functor}, which is a functor $\Gamma: \Derived \to \Derived$ with a natural morphism $\mu:\Gamma \to 1_\Derived$ such that $\Gamma \mu = \mu \Gamma$ and $\Gamma \mu$ is a natural isomorphism. The natural morphism $\mu$ induces an isomorphism
\[\hom_\Derived(\Gamma X,Y) \cong \hom_\Derived(\Gamma X,\Gamma Y) \]

For any functor $F: \Derived \to \Derived$ the \emph{kernel of $F$}, denoted $\Ker F$, is the full subcategory of $\Derived$ whose objects are those satisfying $FX\cong 0$. The \emph{essential image} of $F$ is the full subcategory $\Im F$ consisting of objects $X$ such that $X \cong FY$ for some $Y$. It is easy to see that any localization functor is determined, up to a unique isomorphism, by its essential image. The same goes for colocalization.

It is well known that any localization functor $L$ gives rise to a colocalization functor $\Gamma$ such that for any $X \in \Derived$ there is an exact triangle
\[ \Gamma X \xrightarrow{\mu_X} X \xrightarrow{\eta_X} LX \]
and such that $\Ker L = \Im \Gamma$ and $\Ker \Gamma = \Im L$. Similarly, a colocalization $\Gamma$ gives rise to a localization $L$ having the same properties as above. It follows that a localization (or colocalization) functor is determined also by its kernel.

\subsection*{Dwyer-Greenlees localization and colocalization}
In~\cite{DwyerGreenlees}, Dwyer and Greenlees introduced two types of localizations and showed how to compute them. We recall their definitions.

\begin{definition}
Fix an object $W$ in $\Derived$. A morphism $f:X \to Y$ in $\Derived$ is called a \emph{$W$-equivalence} if $\ext_A^*(W,f)$ is an isomorphism. An object $X$ is \emph{$W$-null} if $\ext_R^*(W,X)=0$, and $X$ is \emph{$W$-torsion} if $\ext_R^*(W,N)=0$ for any $W$-null object $N$. We say $X$ is \emph{$W$-complete} if $\ext_R^*(N,X)=0$ for any $W$-null object $N$. An object $C$ is a \emph{$W$-colocalization} of $X$ if $C$ is $W$-torsion and there is a $W$-equivalence $C \to X$. A complex $N$ is a \emph{$W$-nullification} of $X$ if there is a triangle $C \to X \to N$ such that $C$ is $W$-torsion and $N$ is $W$-null (note that this immediately implies $C$ is a $W$-colocalization of $X$). Finally, an object $C$ is a \emph{$W$-completion} of $X$ if $C$ is $W$-complete and there is a $W$-equivalence $X \to C$. Clearly, $W$-nullification and $W$-completion are both localizations.
\end{definition}

It is easy to see that the localizing subcategory generated by $W$ is contained in the full subcategory of $W$-torsion objects. By arguments of Hirschhorn~\cite{Hirschhorn} there is always a colocalization functor whose essential image is $\loc_\Derived(W)$. It follows that $\loc_\Derived(W)$ is in fact the subcategory of $W$-torsion objects. In other words, $X$ is $W$-torsion if and only if $W$ builds $X$.

When $W$ is compact, Dwyer and Greenlees provide formulas for $W$-completion and $W$-colocalization. Their results are summarized in the following theorem.
\begin{theorem}[Dwyer and Greenlees~\cite{DwyerGreenlees}]
\label{thm: Dwyer-Greenlees colocalization and completion formulas}
Let $W$ be a compact object of $\Derived(R)$ and let $\ee$ be the derived endomorphism $k$-algebra $\End_R(W)$. Then for any object $X \in \Derived$ the natural morphism
\[ \Hom_R(W,X)\otimes_\ee W \to X\]
is a $W$-colocalization of $X$, and the natural morphism
\[X \to  \Hom_{\ee^\op}(\Hom_R(W,R),\Hom_R(W,X)) \]
is $W$-completion. Moreover, the natural morphism $R \to \Hom_{\ee^\op}(\Hom_R(W,R),\Hom_R(W,R))$ is a morphism of $R^e$-modules.
\end{theorem}

For later use we record the following simple lemma.
\begin{lemma}
\label{lem: building modules and completeness}
Let $X$ and $Y$ be $R$-modules. If $R$ is $Y$-complete and $X$ builds $Y$ then $R$ is also $X$-complete.
\end{lemma}
\begin{proof}
Let $N$ be an $X$-null module. Since $X$ builds $Y$ then $N$ is also $Y$-null. Because $R$ is $Y$-complete then $\ext_R^*(N,R)=0$. Hence $R$ is $X$-complete.
\end{proof}

\subsection*{Within the main example}
Recall that the main example we deal with is when $M$ is a simply-connected, compact, manifold, $k$ is the Eilenberg-Mac Lane spectrum of a commutative ring and $R = \chains_* \Omega M$. By results of Dwyer, Greenlees and Iyengar~\cite{DwyerGreenleesIyengar}, $k$ is compact as an $R$-module. Since $M$ satisfies Poincar\'e duality, the results of~\cite{DwyerGreenleesIyengar} also show that $\Hom_R(k,R) \sim \Sigma^{-d} k$. The next two results are well known.

\begin{lemma}
The object $k$ is $k$-complete in $\Derived(R)$.
\end{lemma}
\begin{proof}
Let $\Gamma_k R$ be the $k$-colocalization of $R$. From~\cite{DwyerGreenlees} we learn that $\Gamma_k R$ is in fact an $R$-bimodule and that $k$-colocalization of any $R$-module $X$ is given by the derived tensor product $\Gamma_k R \otimes_R X$. It follows that for any $k$-null module $N$, $\Gamma_k R \otimes_R N = 0$. Since $k$ is built by $\Gamma_k R$, it follows that also $k\otimes_R N =0$. Hence:
\[ \Hom_R(N,k) \simeq \Hom_R(N,\Sigma^d \Hom_R(k,R)) \simeq \Hom_R(k\otimes_R N, \Sigma^d R) = 0\]
\end{proof}

\begin{corollary}
Let $X$ be an $R$-module such that $\pi_i X = 0$ for $i<0$ and $\pi_i X$ is a product of copies of $k$ for every $i\geq 0$. Then $X$ is $k$-complete.
\end{corollary}
\begin{proof}
First, if $\pi_i X$ is non zero only for a single index $i=i_0$, then By~\cite[Proposition 3.9]{DwyerGreenleesIyengar} $X$ is isomorphic in $\Derived(R)$ to a suspension of a product of copies of $k$. It follows from the previous lemma and the definition of completeness that $X$ is $k$-complete. Next, suppose that $\pi_i X$ is non zero only for $0 \leq i \leq n$ for some $n$. Using Postnikov sections and the previous case, an inductive argument on $n$ shows that $X$ is $k$-complete. Finally, suppose that $X$ has infinitely many non-zero homotopy groups. Let $X_n$ be the $n$'th Postnikov section of $X$. Hence there is a natural morphism $X \to X_n$ which is an isomorphism on $\pi_i$ for $i \leq n$ and $\pi_j X_n = 0$ for $j>n$. Since $X$ is the homotopy limit of $X_n$, and every $X_n$ is $k$-complete, then so is $X$.
\end{proof}

Combining the last two results with Lemma~\ref{lem: building modules and completeness} we get the following corollary, which explains the connection of thick subcategories to the question of completeness of $R$.
\begin{corollary}
\label{cor: k in thick means R is complete}
Suppose that $\pi_*R$ is a polynomial ring and let $W$ be a compact $R$-module. If $k \in \thick_{\Derived(R)}W$ then $R$ is also $W$-complete.
\end{corollary}


\section{Hochschild cohomology and Dwyer-Greenlees completion}
\label{sec: Hochschild cohomology and Dwyer-Greenlees completion}

In this section we explain the origin of the natural morphism $\varphi_X: HH^*(R) \to HH^*\End_R(X)$ for an $R$-module $X$. As before, $R$ is a $k$-algebra where $k$ is a commutative $\sphere$-algebra.

\subsection*{Definition of the natural morphism $\varphi$}
We first record the following lemma, which is a simple modification of a result of Lazarev~\cite[Lemma 3.1]{LazarevSpacesOfMultiplicative}. The details of how to deduce Lemma~\ref{lem: tri-modules} from~\cite[Lemma 3.1]{LazarevSpacesOfMultiplicative} are left to the reader.
\begin{lemma}
\label{lem: tri-modules}
Let $A$, $B$ and $C$ be $k$-algebras. Let $X$ be an $A\otimes B^\op$-module, let $Y$ be a $B\otimes C^\op$-module and let $Z$ be a $A\otimes C^\op$-module. Then there is a natural equivalence of $k$-modules:
\[ \Hom_{A \otimes_k C^\op} (X\otimes_B Y, Z) \cong \Hom_{B\otimes_k C^\op}(Y, \Hom_A(X,Z))\]
In particular, there is a natural isomorphism
\[ \ext_{A \otimes_k C^\op}^* (X\otimes_B Y, Z) \cong \ext_{B\otimes_k C^\op}^*(Y, \Hom_A(X,Z))\]
\end{lemma}


To get the morphism $\varphi$ we follow a trick of Koenig and Nagase~\cite{KoenigNagase}.
\begin{lemma}
\label{lem: isomorphisms in varphi}
Let $W$ be a compact $R$-module, let $\ee = \End_R(W)$ and let $W^\natural = \Hom_R(W,R)$. Then there is a natural isomorphism
\[\ext^*_{R^e}(R,\Hom_{\ee^\op}(W^\natural,W^\natural)) \cong \ext^*_{\ee^e}(\ee, \ee)\]
\end{lemma}
\begin{proof}
There are the following isomorphisms
\begin{align*}
\ext^*_{R^e}(R,\Hom_{\ee^\op}(W^\natural,W^\natural)) &\cong \ext^*_{\ee^\op \otimes R^\op}(W^\natural \otimes_R R , W^\natural) \\
&\cong \ext^*_{\ee^\op \otimes R^\op}(W^\natural,W^\natural) \\
&\cong \ext^*_{R^\op \otimes \ee^\op}(W^\natural \otimes_\ee \ee, W^\natural) \\
&\cong \ext^*_{\ee^e}(\ee, \Hom_{R^\op}(W^\natural,W^\natural)) \\
&\cong \ext^*_{\ee^e}(\ee,\ee)
\end{align*}
Except for the last line, all the isomorphisms follow from Lemma~\ref{lem: tri-modules}. The last line follows from the fact that $\Hom_{R^\op}(W^\natural,W^\natural)$ is equivalent to $\ee = \Hom_R(W,W)$, which we now explain. Since $W$ is compact, so is $W^\natural$. Moreover, compactness implies that $W \simeq \Hom_{R^\op}(W^\natural,R^\op)$. Hence:
\begin{align*}
\Hom_{R^\op}(W^\natural,W^\natural) &\simeq \Hom_{R^\op}(W^\natural,R^\op) \otimes_{R^\op} W^\natural \\
&\simeq W^\natural \otimes_R \Hom_{R^\op}(W^\natural,R^\op) \\
&\simeq W^\natural \otimes_R W\\
&\simeq \Hom_R(W,R)\otimes_R W\\
&\simeq \Hom_R(W,W)
\end{align*}
\end{proof}

\begin{definition}
\label{def: definition of varphi}
For a compact $R$-module $W$ define $\varphi_W$ to be the composition:
\[ \ext^*_{R^e} (R,R) \to \ext^*_{R^e}(R,\Hom_{\ee^\op}(W^\natural,W^\natural)) \cong \ext^*_{\ee^e}(\ee, \ee)\]
where the leftmost morphism is induced by the Dwyer-Greenlees natural completion morphism $R \to \Hom_{\ee^\op}(W^\natural,W^\natural)$.
\end{definition}

\subsection*{Multiplicative property of $\varphi$}
As always, showing that a map is multiplicative is a complicated process. We start with a simple discussion. Let $A$ and $B$ be $k$-algebras and let $X$ be an $A \otimes B$-module, then $\Hom_B(X,X)$ is an $A$-bimodule. Moreover, $\Hom_B(X,X)$ is an algebra, where the multiplication map is the map induced by composition. Since $A$ is a coalgebra in $\Derived(A^e)$, then $\ext^*_{A^e}(A,\Hom_B(X,X))$ is a graded ring. We make this structure explicit in the following definition.

\begin{definition}
Define a pairing $\Psi:\ext^*_{A^e}(A,\Hom_B(X,X)) \times \ext^*_{A^e}(A,\Hom_B(X,X)) \to \ext^*_{A^e}(A,\Hom_B(X,X))$ in the following manner. Given maps $f:\Sigma^n A \to \Hom_B(X,X)$ and $g: \Sigma^m A \to \Hom_B(X,X)$ let $\Psi(f\times g)$ be the composition:
\[ \Sigma^{n+m}A \xrightarrow{\nabla} \Sigma^n A \otimes_A \Sigma^m A \xrightarrow{f\otimes g} \Hom_B(X,X)\otimes_A \Hom_B(X,X) \xrightarrow{\theta} \Hom_B(X,X)\]
Where $\nabla$ is the inverse of the multiplication map $A\otimes_A A \to A$ and $\theta$ is the composition (or multiplication) map. It is left to the reader to verify that $\Psi$ is associative and bilinear. The next lemma will also show that $\Psi$ is also unital.
\end{definition}

\begin{lemma}
\label{lem: general multiplicativity lemma}
The isomorphism
\[ \ext^*_{A^e}(A,\Hom_B(X,X)) \cong \ext^*_{A \otimes B}(X, X)\]
coming from Lemma~\ref{lem: tri-modules} cmmutes with the respective pairings (multiplication on the right hand side and $\Psi$ on the left hand side). Therefore the pairing $\Psi$ also has a unit.
\end{lemma}
\begin{proof}
Let $f$ and $g$ be two morphisms $A \to \Hom_B(X,X)$ in $\Derived(A^e)$. The isomorphism in Lemma~\ref{lem: tri-modules} describes an adjunction between two functors. Hence we can identify the image of $f$ and $g$ under this isomorphism explicitly. Recall that the adjoint morphism $\hat{f}$ is given by the composition $\epsilon \circ (f\otimes 1_X) \circ u_X$ where $\epsilon:\Hom_B(X,X) \otimes_A X \to X$ is the evaluation map and $u_X: X \to A\otimes_A X$ is the inverse of the unit isomorphism. The product $\hat{f} \cdot \hat{g}$ is the composition of the two morphisms and is therefore given by
\[\hat{f} \cdot \hat{g} = \epsilon (f\otimes 1_X) u_X \epsilon (g \otimes 1_X) u_X \]
It is easy to see that this product is also equal to:
\[ \epsilon (1_{\Hom_B(X,X)} \otimes \epsilon) (f\otimes g \otimes 1_X) u_{A\otimes_A X}  u_X \]

Recall that the evaluation map is in fact isomorphic to the composition map: $\Hom_B(X,X) \otimes_A \Hom_B(B,X) \to \Hom_B(B,X)$. Since composition is associative then
\[ \epsilon (1_{\Hom_B(X,X)} \otimes \epsilon) = \epsilon (\theta \otimes 1_X)\]
We also note that $u_{B\otimes_B X}  u_X = (\nabla \otimes 1_X) u_X$, leaving it to the reader to verify. Altogether we have:
\begin{align*}
\hat{f} \cdot \hat{g} &=  \epsilon (\theta \otimes 1_X) (f\otimes g \otimes 1_X) (\nabla \otimes 1_X) u_X \\
&= \epsilon (\Psi(f\times g) \otimes 1_X) u_X \\
&= \hat{\Psi}(f \times g)
\end{align*}

Since the identity map $X \to X$ is a unit for $\ext^*_{A \otimes B}(X, X)$, its adjoint $A \to \Hom_B(X,X)$ is a unit for $\Psi$, showing that $\ext^*_{A^e}(A,\Hom_B(X,X))$ is a graded ring.
\end{proof}

\begin{proposition}
\label{pro: varphi is multiplicative}
Let $W$ be a compact $R$-module. Then the map $\varphi_W$ is a map of graded rings.
\end{proposition}
\begin{proof}
Recall that in Lemma~\ref{lem: isomorphisms in varphi} we used the following isomorphisms:
\begin{align*}
\ext^*_{R^e}(R,\Hom_{\ee^\op}(W^\natural,W^\natural)) &\cong \ext^*_{\ee^\op \otimes R^\op}(W^\natural \otimes_R R , W^\natural) \\
&\cong \ext^*_{\ee^\op \otimes R^\op}(W^\natural,W^\natural) \\
&\cong \ext^*_{R^\op \otimes \ee^\op}(W^\natural \otimes_\ee \ee, W^\natural) \\
&\cong \ext^*_{\ee^e}(\ee, \Hom_{R^\op}(W^\natural,W^\natural))
\end{align*}
These are all multiplicative by Lemma~\ref{lem: general multiplicativity lemma}. It remains to show that the isomorphism
\[ \ext^*_{\ee^e}(\ee, \Hom_{R^\op}(W^\natural,W^\natural)) \cong \ext^*_{\ee^e}(\ee, \ee)\]
and the map
\[\ext^*_{R^e} (R,R) \to \ext^*_{R^e}(R,\Hom_{\ee^\op}(W^\natural,W^\natural))\]
are multiplicative. For the former, note that the multiplication on $\ext^*_{\ee^e}(\ee, \ee)$ induced by the isomorphism from $\ext^*_{\ee^e}(\ee, \Hom_{R^\op}(W^\natural,W^\natural))$, is such that the product of two morphisms $f,g\in \ext^*_{\ee^e}(\ee, \ee)$ is the tensor product $f \otimes_\ee g$. By the Eckmann-Hilton argument, this is the same as the standard ring structure on $\ext^*_{\ee^e}(\ee,\ee)$ (which is given by composition).

For the latter, we again use the Eckmann-Hilton argument. Given $f,g \in \ext^*_{R^e} (R,R)$, the Eckmann-Hilton argument shows that the composition of these two morphisms is equal to their tensor product $f \otimes_R g$ (more precisely to $\nabla^{-1} f\otimes_R g \nabla$). Let $\gamma : R \to \Hom_{\ee^\op}(W^\natural,W^\natural)$ be the Dwyer-Greenlees $W$-completion map. Hence we need to show that
\[\gamma \nabla^{-1} f\otimes_R g \nabla\] is equal to
\[\mu (\gamma \otimes_R \gamma) (f \otimes_R g) \nabla\]
where $\mu : \Hom_{\ee^\op}(W^\natural,W^\natural) \otimes_R \Hom_{\ee^\op}(W^\natural,W^\natural) \to \Hom_{\ee^\op}(W^\natural,W^\natural)$ is the multiplication map of $\Hom_{\ee^\op}(W^\natural,W^\natural)$ (as an $R$-algebra). We see that it suffices to show that the Dwyer-Greenlees completion map $\gamma$ is a map of $k$-algebras.


More generally, suppose that $A$ and $B$ are $k$ algebras and that $X$ is an $A \otimes B$-module. To complete the proof it suffices to show that the natural map $A \to \Hom_B(X,X)$ is a map of $k$-algebras. We leave this last statement to the reader to verify.
\end{proof}

\subsection*{Within the main example}
As before $M$ is a simply-connected, compact, manifold, $k$ is the Eilenberg-Mac Lane spectrum of a commutative ring and $R = \chains_* \Omega M$.

\begin{proposition}
\label{pro: when varphi is an iso}
Suppose that $\pi_*R$ is a polynomial ring over $\pi_0 R$. Let $W$ be a compact $R$-module. If $k \in \thick_{\Derived(R)}W$ then
\[\varphi_W: H_{*+m}(\fl M;k) \cong HH^*(R) \to HH^*(\End_R(W))\]
is an isomorphism of graded rings.
\end{proposition}
\begin{proof}
This is a simple application of Corollary~\ref{cor: k in thick means R is complete} and the properties of $\varphi$. Note that we are using here Malm's isomorphism $H_{*+m}(\fl M;k) \cong HH^*(R)$~\cite{Malm}.
\end{proof}

Thus, to understand when $\varphi$ is an isomorphism we need to study the partially ordered set of thick subcategories in $\Derived(R)^\compact$.

\section{Stratification of a derived category}
\label{sec: Stratification of a derived category}

In this section we recall the definitions and properties of stratification given in~\cite{BIKstratifying}. Throughout this section $\Derived$ denotes the derived category of a $k$-algebra $R$, where $k$ is a commutative $\sphere$-algebra. Hence $\Derived$ is a triangulated category which has arbitrary coproducts and a compact generator. This last property allows us to simplify some of the definitions of Benson, Iyengar and Krause from~\cite{BIKstratifying}, and only the simplified versions will be given.

\subsection*{Spectrum and support}
The \emph{center} of $\Derived$ is the graded-commutative ring of natural transformations $\alpha:1_\Derived \to \Sigma^n$ satisfying $\alpha \Sigma = (-1)^n \Sigma \alpha$. Following~\cite{BIKstratifying} we say that a Noetherian graded-commutative ring $S$ \emph{acts on $\Derived$} if there is a homomorphism of graded rings from $S$ to the graded-commutative center of $\Derived$. Note that the action of $S$ on $\Derived$ naturally makes $\hom_\Derived(X,Y)$ into an $S$-module, for any two objects $X$ and $Y$ in $\Derived$. In particular, $\pi^*X$ is an $S$-module for every $X\in \Derived$.

For example, the Hochschild cohomology $HH^*(R)$ acts on $\Derived$ in the following manner: given $f:\Sigma^n R \to R$ the natural transformation $f\otimes_R -$ is an element of the center of $\Derived$. Of course, the Hochschild cohomology need not be Noetherian. One solution to this problem is to simply choose a Noetherian subring of $HH^*(R)$.

Let $\spec S$ be the partially ordered set of homogeneous prime ideals of $S$. A subset $\vv \subset \spec S$ is \emph{specialization closed} if whenever $\pp \subset \qq$ and $\pp \in \vv$ then also $\qq \in \vv$. Given a specialization closed subset $\vv$, an object $X\in \Derived$ is called \emph{$\vv$-torsion} if $(\pi^*X)_\pp = 0$ for all $\pp \not\in \vv$, where $(\pi^*X)_\pp$ is the usual (homogeneous) localization at $\pp$. Let $\Derived_\vv$ be the full subcategory of $\vv$-torsion objects, this is a localizing subcategory \cite[Lemma 4.3]{BIKsupport}. By \cite[Proposition 4.5]{BIKsupport} there is a localization $L_\vv$ and a colocalization $\Gamma_\vv$ such that
\[ \Gamma_\vv X \to X \to L_\vv X \]
is an exact triangle and $\Ker L_\vv = \Im \Gamma_\vv = \Derived_\vv$. Note that both $\Gamma_\vv$ and $L_\vv$ are \emph{smashing}, i.e. they commute with coprducts~\cite[Corollary 6.5]{BIKsupport}.

\begin{remark}
We are using here the algebraic definition of a smashing localization, the definition used in algebraic topology is different. In topology a localization functor $L$ in a symmetric monoidal category is \emph{smashing} if it is equivalent to the functor $A \otimes -$ for some object $A$, where $\otimes$ is the monoidal product. Since we will not be working over a symmetric monoidal triangulated category in this paper, we shall only use the algebraic definition.
\end{remark}

Given a prime ideal $\pp \in \spec S$ let $\zz(\pp)=\{ \qq \in \spec S \ | \ \qq \nsubseteq \pp \}$. For a homogeneous ideal $\aa$ of $S$ we let $\vv(\aa)=\{\qq \in \spec S \ | \ \aa \subseteq \qq \}$. Both $\zz(\pp)$ and $\vv(\aa)$ are specialization closed subsets of $\spec S$. For $X$ in $\Derived$ denote by $X_\pp$ the localization $L_{\zz(\pp)}X$. By~\cite[Theorem 4.7]{BIKsupport} $\pi^*(X_\pp) \cong (\pi^* X)_\pp$ as $S$-modules, which justifies this notation. From this it is easy to see that $S_\pp$ naturally acts on the localizing subcategory $\Im L_{\zz(\pp)}$, in a way which extends the action of $S$.

\begin{definition}
\label{def: support according to BIK}
Let $\Gamma_\pp$ be the exact functor given by
\[ \Gamma_\pp X = \Gamma_{\vv(\pp)} X_\pp\]
It turns out that $\Im \Gamma_\pp$ is a localizing subcategory of $\Derived$ (see~\cite{BIKstratifying}). For an object $X$ in $ \Derived$ the \emph{support} of $X$ over $S$ is defined to be
\[ \supp_S X = \{ \pp \in \spec S \ | \ \Gamma_\pp X \neq 0 \} \]
\end{definition}

\subsection*{Stratification}
As above, we assume that the Noetherian ring $S$ acts on $\Derived$. To define stratification we need two more concepts from~\cite{BIKstratifying}. First, a localizing subcategory of $\Derived$ is \emph{minimal} if it is nonzero and contains no nonzero localizing subcategories. Second, a \emph{local-global principle holds} for the action of $S$ on $\Derived$ if for every $X\in \Derived$
\[ \loc_\Derived (X) = \loc_\Derived (\{ \Gamma_\pp X \ | \ \pp \in \spec S\} )\]
For example, by~\cite[Corollary 3.5]{BIKstratifying} if $S$ has a finite Krull dimension then the local-global principle holds.

\begin{definition}[\cite{BIKstratifying}]
\label{def: Stratification}
The triangulated category $\Derived$ is \emph{stratified} by $S$ if the following two conditions hold:
\begin{description}
\item[S1] The local-global principle holds for the action of $S$ on $\Derived$.
\item[S2] For every $\pp \in \spec S$ the localizing subcategory $\Im \Gamma_\pp$ is zero or minimal.
\end{description}
\end{definition}

Then main use of stratification is for classifying localizing and thick subcategories. To present this, we first extend the definition of support to subcategories. For a subcategory $\bl$ of $\Derived$ let $\supp_S \bl$ be the union $\bigcup_{X\in \bl} \supp_S X$. The next theorem is from~\cite{BIKstratifying} (it has been translated to our setting):

\begin{theorem}[{\cite[Theorem 6.1]{BIKstratifying}}]
\label{the: BIK theorem on thick subcats}
Suppose that $\Derived$ is stratified by the action of $S$ and that $\pi^*R$ is finitely generated as an $S$-module. Then the map
\[ \left\{
     \begin{array}{c}
        \text{Thick}\\
        \text{subcategories of }\Derived^\compact\\
     \end{array}
   \right\}
   \xrightarrow{\supp_S(-)}
   \left\{
     \begin{array}{c}
        \text{Specialization closed}\\
        \text{subsets  of }\supp_S \Derived\\
     \end{array}
   \right\}
   \]
is a bijection which respects inclusions. The inverse map sends a specialization closed subset $\vv \subset \supp_S \Derived$ to the full subcategory whose objects are $\{X\in \Derived^\compact \ | \ \supp_S X \subset \vv\}$.
\end{theorem}

The definition given here of support is a bit too cumbersome to be useful. Fortunately results of Benson, Iyengar and Krause from~\cite{BIKsupport} show that we can sometimes use a more familiar notion of support.
\begin{lemma}
\label{lem: familiar support}
Suppose that $\pi^*R$ is finitely generated as an $S$-module. Let $X$ be any compact object in $\Derived$, then $\supp_S X = \supp_S \pi^*X$, where $\supp_S \pi^*X$ is the usual support of a graded module over a graded-commutative ring. Moreover, in this case
\[\supp_S \thick_\Derived X = \supp_S \pi^*X\]
\end{lemma}
\begin{proof}
Since $\pi^*R$ is a finitely generated $S$-module, and $X$ is finitely built by $R$, we conclude that $\pi^*X$ is also a finitely generated $S$-module. By~\cite[Theorem 5.4]{BIKsupport} this implies that $\supp_S X$ is equal to the \emph{small} support of $\pi^*X$ as an $S$-module (see~\cite{BIKsupport} for the definition of the small, or cohomological, support). Since $H^*X$ is a finitely generated module, its small support is the same as the usual support, i.e the set $\{\pp \in \spec S \ | \ (\pi^*X)_\pp \neq 0 \}$.

Let $\Tri$ be the full subcategory of $\Derived^\compact$ consisting of objects $Y$ such that $\supp_S Y \subset \supp_S X$. It is easy to see that $\Tri$ is a thick subcategory which contains $X$. Hence $\thick_\Derived X \subset \Tri$ and the result follows.
\end{proof}

\subsection*{Koszul objects}
Given an object $M$ of $\Derived$ and an element $s \in S$ let $M/s$ be the object defined by the following exact triangle
\[ \Sigma^{-|s|} M \xrightarrow{s_M} M \to M/s \]
Obviously, the object $M/s$ is defined only up to isomorphism. For a sequence $s=(s_1,...,s_n) \subset S$ let $M/s$ be the object $M/s_1/s_2\cdots/s_n$. Note that there is an obvious morphism $M \to M/s$, which we will use often.

Let $I$ be an ideal of $S$. We denote by $M/I$ any object of the form $M/(s_1,...,s_n)$ where $s_1,...,s_n$ are generators of $I$. Certain properties of $M/I$ do not depend on the particular choice of generators (see e.g.~\cite[Lemma 2.6]{BIKstratifying}), which justifies the notation.

\section{Hochschild cohomology and localization}
\label{sec: Hochschild cohomology and localization}


This section has two goals. First, we need to show that certain localizations of $R$ are bimodules. Next, we show that the endomorphisms (as bimodules) of these localizations are the localization of the Hochschild cohomology of $R$. This will come in handy in the next section, when we use this fact to deduce that certain Koszul constructions result in bimodules.

Throughout this section $R$ is a $k$-algebra where $k$ is a commutative $\sphere$-algebra and $\Hring$ is a Noetherian subring of the Hochschild cohomology ring $\ext^*_{R^e}(R,R)$. There is a \emph{left} action of $\Hring$ on $\Derived(R^e)$ which comes from utilizing the left action of $R$ on $R$-bimodules. Explicitly, given an element $h \in \Hring^n$ and an $R$-bimodule $B$ the action of $h$ on $B$ is the map $h \otimes_R 1_B$. One should bear in mind that the right action will usually be very different from the left one. Here we will only consider the left action.

Let $\uu$ be a specialization closed subset of $\spec \Hring$. Recall that $\Derived(R^e)_\uu$ is the localizing subcategory of $\uu$-torsion bimodules. This gives rise to the localization functor which we denote by $L^e_\uu$ - the unique localization functor whose kernel is $\Derived(R^e)_\uu$. On the other hand we can consider the localizing subcategory $\Derived(R)_\uu$ and its associated localization $L_\uu: \Derived(R) \to \Derived(R)$.

\begin{theorem}
\label{thm: Bimodule localization}
Let $\uu$ be a specialization closed subset of $\spec \Hring$. Then the localization functor $L_\uu$ is naturally isomorphic to the augmented functor $L_\uu^e R \otimes_R -$.
\end{theorem}
\begin{proof}
By~\cite[Theorem 6.4]{BIKsupport} the localizing subcategory $\Derived(R)_\uu$ is generated by the set of compact objects $R/ \pp$ for all $\pp \in \uu$. Similarly, $\Derived(R^e)_\uu$ is generated by the set of compact objects
\[ \{ R/ \pp \otimes_R R^e \ | \ \pp \in \uu\} \]
Note that here we are implicitly using the fact that $R/\pp$ is naturally a bimodule.

Let $u:k \to R$ be the unit map of $R$ and let $\lambda: R\cong R\otimes k \to R^e$ be $1_R \otimes u$. Given an $R$-module $M$ and a prime ideal $\pp \in \uu$, we have the following isomorphisms:
\begin{align*}
\ext_R^*(R/\pp, L_\uu^e R  \otimes_R M) &\cong \ext_R^*(R/\pp,R)\otimes_R L_\uu^e R  \otimes_R M \\
   &\cong \ext_R^*(R/\pp, L_\uu^e R ) \otimes_R M \\
   &\cong \ext_{R^e}^*(R/\pp \otimes_R R^e, L_\uu^e R ) \otimes_R M = 0
\end{align*}
The first and second isomorphisms use the compactness of $R/\pp$, while the third uses the usual adjunction arising from $\lambda$. The upshot of the calculation above is that for any object $X$ in $\Derived(R)_\uu$ we have $\ext^*_R(X,L_\uu^e R  \otimes_R M) = 0$. This implies that $\Gamma_\uu (L_\uu^e R  \otimes_R M) = 0$ and therefore $L_\uu^e R  \otimes_R M$ is in the image of $L_\uu$.

The next step is to show that $\Gamma_\uu^e(R) \otimes_R M \in \Derived(R)_\uu$. Consider $\Gamma^e_\uu$ - the colocalization functor whose essential image is $\Derived(R^e)_\uu$. The bimodule $\Gamma_\uu^e(R)$ is built by the set of bimodules $ \{ R/ \pp \otimes_R R^e \ | \ \pp \in \uu\}$. Hence, as a left module (via $\lambda$) $\Gamma_\uu^e(R)$ is also built by the modules $ \{ R/ \pp \otimes_R R^e \ | \ \pp \in \uu\}$.

As left $R$-modules, $R$ builds $R^e$. By tensoring this recipe on the left with $R/\pp$ we see that $R/ \pp \otimes_R R^e$ is built by $R/ \pp$ in $\Derived(R)$. It follows that $\Gamma_\uu^e(R) \in \Derived(R)_\uu$. Now consider the functor $\bar{\Gamma} = \Gamma_\uu^e(R) \otimes_R- :\Derived(R) \to \Derived(R)$. Since $\bar{\Gamma}$ preserves coproducts, it is easy to see that $\bar{\Gamma}^{-1}(\Derived(R)_\uu)$ is a localizing subcategory which contains $R$, hence $\bar{\Gamma}^{-1}(\Derived(R)_\uu) = \Derived(R)$. We conclude that $\Gamma_\uu^e(R) \otimes_R M \in \Derived(R)_\uu$ for any $M$.

By applying the functor $L_\uu$ to the exact triangle $\Gamma_\uu^e(R) \otimes_R M \to M \to L_\uu^e R  \otimes_R M$ we see that the functor $L_\uu^e R  \otimes_R -$ is isomorphic to $L_\uu$.
\end{proof}

Before continuing we note the following property. The morphism $R \to L_\uu^e R$ induces a map:
\[ l: \ext^*_{R^e}(R,R) \to \ext^*_{R^e}(R,L_\uu^e R) \cong \ext^*_{R^e}(L_\uu^e R,L_\uu^e R)\]
It is easy to see that this is a map of rings (where multiplication is given by composition).

\begin{corollary}
\label{cor: Hochschild cohomology of localization}
Suppose that $R$ is a compact $R^e$-module. Let $\pp \subset \Hring$ be a prime ideal and let $\uu = \zz(\pp)$. Then there is as isomorphism of rings over $\ext^*_{R^e}(R,R)$:
\[ \ext^*_{R^e} (L_\uu^e R ,L_\uu^e R ) \cong \ext^*_{R^e}(R,R)_\pp\]
In addition, if the map $\ext^*_{R^e}(R,R) \to \pi^*(R)$ is surjective then the natural map $\ext^*_{R^e} (L_\uu^e R ,L_\uu^e R ) \to \pi^*L_\uu^e R $ is also surjective.
\end{corollary}
\begin{proof}
By~\cite[Proposition 2.3]{BIKstratifying} the map $R \to L_\uu^e R$ induces an isomorphism
\[ \ext^*_{R^e}(R,R)_\pp \cong \ext^*_{R^e} (R ,L_\uu^e R ) \]
By the properties of localization the map $R \to L_\uu^e R$ induces an isomorphism
\[ \ext^*_{R^e} (L_\uu^e R ,L_\uu^e R ) \cong \ext^*_{R^e} (R ,L_\uu^e R ) \]
From naturality of the relevant maps (see~\cite[Proposition 2.3]{BIKstratifying}) we get a commutative diagram
\[ \xymatrixcompile{
{\ext^*_{R^e}(R,R)_\pp} \ar[d] \ar[r]^\cong & \ext^*_{R^e} (R ,L_\uu^e R ) \ar[d]  & {\ext^*_{R^e} (L_\uu^e R ,L_\uu^e R )} \ar[l]^\cong \ar[d] \\
{\pi^*(R)_\pp} \ar[r]^\cong & {\pi^*(L_\uu^e R)} & {\ext^*_R(L_\uu^e R ,L_\uu^e R )} \ar[l]^\cong
}\]
To explain the isomorphisms at the bottom of the diagram, note that Theorem~\ref{thm: Bimodule localization} shows that $L_\uu R \cong L_\uu^e R$ as $R$-modules (with the $R$-module structure on $L_\uu^e R$ is induced via $\lambda$). Now the bottom isomorphisms follow from the same reasons as the top ones. Since the left-most vertical map is a surjection (localization at a prime is exact), then so is the right-most vertical map.
\end{proof}

\section{Left ring objects}
\label{sec: Left ring objects}

Left ring objects are bimodules that share a small amount of the structure of rings, though one must bear in mind that these are not rings in any standard sense. Still, their structure is enough to make them useful for the purpose of understanding the localizing subcategories they generate. We remark that this section is, in essence, a non-commutative generalization of~\cite[Section 4]{ShamirStratify}.

Throughout this section, as before, $R$ is a $k$-algebra, $k$ is a commutative $\sphere$-algebra and $\Hring$ is a Noetherian subring of $\ext^*_{R^e}(R,R)$.

\subsection*{Left ring objects}
A \emph{left ring object} of $R$ is an $R$-bimodule $A$  endowed with a unit morphism $u:R \to A$ of $R$-bimodules and a product morphism $m:A \otimes_R A \to A$ of \emph{left} $R$-modules, such that $m(u\otimes 1_A)$ is the natural isomorphism $\ell_A:R\otimes_R A \to A$. A morphism of left ring objects is a morphism in $\Derived(R^e)$ which respects the relevant structures. We emphasize that there is no assumption of commutativity nor associativity of $A$. A left \emph{$A$-module} is an object $M\in \Derived(R)$ endowed with a morphism $a:A\otimes_R M \to M$ satisfying $a(u \otimes 1_M) = \ell_M$. Again, we do not assume associativity of this structure. Note that the definition of a left ring object is more general than the definition of a ring object given in~\cite[Definition V.2.1]{EKMM}.

The next lemma is the source of usefulness of left ring objects.

\begin{lemma}
\label{lem: Ring object maps surject module homotopy}
Let $A$ be a left ring object of $R$ and let $M$ be an $A$-module. Then the map
\[ u^*:\ext_R^*(A,M) \to \pi^*M\]
induced by the unit $u:R\to A$ is a surjection.
\end{lemma}
\begin{proof}
The proof is essentially the same as in~\cite[Lemma 4.1]{ShamirStratify}.
\end{proof}

\subsection*{Left ring objects arising from localizations}
For the rest of this section $\uu$ denotes a specialization closed subset of $\spec \Hring$, $A$ is the bimodule $L_\uu^e R$ and $u$ is the natural map of bimodules $R \to A$. We will use the results of Section~\ref{sec: Hochschild cohomology and localization} to show that $A$ is a left ring object.

\begin{lemma}
In the situation above there exists a morphism $m:A \otimes_R A \to A$ of bimodules which makes $A$ into a left ring object.
\end{lemma}
\begin{proof}
From Theorem~\ref{thm: Bimodule localization} it follows that $u \otimes_R A: A \to A\otimes_R A$ is an isomorphism of left modules. Since $u \otimes_R A$ is a morphism of bimodules, it is also an isomorphism of bimodules. Taking $m:A \otimes_R A \to A$ to be the inverse of $u \otimes_R A$ completes the proof.
\end{proof}

Theorem~\ref{thm: Bimodule localization} also implies the following strengthening of Lemma~\ref{lem: Ring object maps surject module homotopy}.
\begin{lemma}
\label{lem: Localization ring objects bijects homotopy}
Let $M$ be an $A$-module, then $M \cong A \otimes_R M$ and the map
\[ u^*:\ext_R^*(A,M) \to \pi^*M\]
is an isomorphism.
\end{lemma}
\begin{proof}
By the definition of an $A$-module, $M$ is a retract of $A \otimes_R M$. In other words, $M$ is a retract of $L_\uu M$. Therefore, $L_\uu M$ is isomorphic to the direct sum $M \oplus \Sigma \Gamma_\uu M$ and in particular $\Sigma \Gamma_\uu M$ is a retract of $L_\uu M$. Since
\[ \ext^*_R(\Gamma_\uu M, L_\uu M)= \ext^*_R(L_\uu \Gamma_\uu M, L_\uu M) = 0,\]
we conclude that $\Gamma_\uu M = 0$. It follows that the localization map $M \to L_\uu M$ is an isomorphism. Hence the map $R \to L_\uu R$ induces an isomorphism
\[ \ext_R^*(L_\uu R,M) \to \ext_R^*(R,M). \]
\end{proof}

It is easy to see that the category of $A$-modules is precisely the localizing subcategory $\Im L_\uu$ of $\Derived(R)$. It follows that there is a natural action of $\ext^*_{R^e}(A,A)$ on $\Im L_\uu$: given $f\in \ext^*_{R^e}(A,A)$ and an $A$-module $M$ we use the morphism
\[ \Sigma^{|f|} M \cong \Sigma^{|f|} A \otimes_R M \xrightarrow{f \otimes 1} A \otimes_R M  \cong M\]
As we saw in Section~\ref{sec: Hochschild cohomology and localization}, there is a natural map of rings
\[ l:\ext^*_{R^e}(R,R) \to \ext^*_{R^e}(A,A)\]
It now appears that there might be two actions of $\ext^*_{R^e}(R,R)$ on $\Derived(R)$ - the original action, and an action induced by $l$. Fortunately, both actions are the same. We leave the proof of this fact to the reader.

Next, suppose that $R$ is a compact $R^e$-module and that $\uu = \zz(\pp)$ for some prime ideal $\pp \subset \Hring$. On the one hand, for any $A$-module $M$ we have that $\pi^*M = (\pi^*M)_\pp$ as $\Hring$-modules~\cite[Theorem 4.7]{BIKsupport}. Hence the action of every element in $\Hring \setminus \pp$ on $M$ is invertible. This implies that the triangulated category $\Im L_\uu$ has a natural action of $\Hring_\pp$ which extends the original action of $\Hring$.

On the other hand, Corollary~\ref{cor: Hochschild cohomology of localization} implies that $l$ induces a map $l_\pp:\Hring_\pp \to \ext^*_{R^e}(A,A)$. This gives a second action of $\Hring_\pp$ on $\Im L_\uu$, via $l_\pp$ and the action of $\ext^*_{R^e}(A,A)$ on $\Im L_\uu$. Because the action of $\ext^*_{R^e}(R,R)$ on $\Im L_\uu$ splits through $l$, one readily sees that this second action of $\Hring_\pp$ must be the same as the first one.

\subsection*{Left ring objects arising from regular sequences}
We will need a modified version of the following result from~\cite{EKMM}.

\begin{theorem}[{\cite[Theorem V.2.6]{EKMM}}]
Let $R$ be a commutative $\sphere$-algebra and suppose that $\pi^*R$ is concentrated in even degrees. Let $x$ be a regular element in $\pi^*R$. Then $R/x$ has a left ring object structure such that the obvious morphism $u_x:R \to R/x$ is the unit morphism.
\end{theorem}

For the rest of this section, $A$ is the left ring object $L_\uu^e R$, where $\uu = \zz(\pp)$ for some prime ideal $\pp \subset \Hring$. The modified version of~\cite[Theorem V.2.6]{EKMM} is given below. We emphasize that $R$ is not assumed to be commutative.

\begin{proposition}
\label{pro: Dividing by regular element gives ring object}
Let $A \to B$ be a morphism of left ring objects and let $x$ be an element of $\ext^*_{R^e}(A,A)$. Suppose that $\pi^*B$ is concentrated in even degrees and that the action of $x$ on $\pi^*B$ is regular. Then $B/x$ has a left ring object structure such that the composition $R \to B \to B/x$ is its unit.
\end{proposition}
\begin{proof}
Note that the morphism $A \to B$ makes $B$ into an $A$-module. This implies that $B \to A \otimes_R B$ is an isomorphism of $R$-bimodules. 
From this we conclude that $x$ acts on $B$ by a map of bimodules and hence $B/x$ is a bimodule.  The rest of the proof is practically the same as that of~\cite[Theorem V.2.6]{EKMM}, apart from the fact that we can only prove a left unit relation.
\end{proof}


\begin{corollary}
\label{cor: Dividing by a regular sequence may yield a ring object}
Suppose that $\pi^*A$ is concentrated in even degrees. Let $X=(x_1,...,x_n)$ be a sequence of elements in $\ext^*_{R^e}(A,A)$  which is a $\pi^*A$-regular sequence. Then $A/X$ has a ring object structure such that the composition $R \to A \to A/X$ is the unit morphism.
\end{corollary}
\begin{proof}
The proof in a simple induction on $n$ using Proposition~\ref{pro: Dividing by regular element gives ring object}.
\end{proof}


\subsection*{Module structure}
We first need the following result, which is a minor modification of~\cite[Lemma V.2.4]{EKMM} (the only thing modified is, in fact, the setup).
\begin{lemma}
Let $A \to B$ be a morphism of left ring objects and let $x$ be an element of $\ext^n_{R^e}(A,A)$. Suppose that $B/x$ has a left ring object structure such that obvious morphism $B \to B/x$ is morphism of left ring objects. Then the morphism $x:\Sigma^n B/x \to B/x$ in $\Derived(R)$ is zero.
\end{lemma}
\begin{proof}
The proof is the same as that of~\cite[Lemma V.2.4]{EKMM}.
\end{proof}

\begin{lemma}
\label{lem: Multiplying ring-object by last devided element is zero}
Let $A \to B$ be a morphism of left ring objects and let $y$ be an element of $\ext^n_{R^e}(A,A)$ such that the morphism $y:\Sigma^n B \to B$ in $\Derived(R)$ is zero. Suppose that $B/x$ has a left ring object structure such that obvious morphism $B \to B/x$ is morphism of left ring objects. Then the morphism $y:\Sigma^n B/x \to B/x$ in $\Derived(R)$ is also zero.
\end{lemma}
\begin{proof}
Since $B \to B/x$ is morphism of left ring objects, then $B/x$ is a $B$-module. In particular, $B/x$ is a retract of $B \otimes_R B/x$ in the category of $A$-modules. The result follows.
\end{proof}

\begin{corollary}
\label{cor: Multiplying ring Koszul object by ideal elements is zero}
Suppose that $\pi^*A$ is concentrated in even degrees. Let $X=(x_1,...,x_n)$ be a sequence of elements in $\ext^*_{R^e}(A,A)$  which is a $\pi^*A$-regular sequence. Then for every element $x$ in the ideal generated by $X$, the morphism $x:\Sigma^nA/X \to A/X$ is zero.
\end{corollary}
\begin{proof}
From Lemma~\ref{lem: Multiplying ring-object by last devided element is zero} we see that the statement is true for $x\in \{x_1,...,x_n\}$. The result follows.
\end{proof}

An immediate consequence of the last corollary is the following.

\begin{corollary}
\label{cor: Homotopy of M/X is associative module}
Suppose that $\pi^*A$ is concentrated in even degrees. Let $X=(x_1,...,x_n)$ be a sequence of elements in $\ext^*_{R^e}(A,A)$  which is a $\pi^*A$-regular sequence. Then for every $A$-module $M$ the $\pi^*A$-module $\pi^*(M/X)$ is naturally a $(\pi^*A)/(X)$-module.
\end{corollary}

\subsection*{Regular local rings}
We now consider the case where $\pi^*A$ is a graded commutative local ring concentrated in even degrees. Note that the action of $\ext_{R^e}^*(A,A)$ on $A$ induces is a natural map
\[ \alpha: \ext_{R^e}^*(A,A) \to \ext^*_R(A,A) = \pi^*A.\]
Let $x_1,...,x_n\in \ext^*_{R^e}(A,A)$ be a $\pi^*A$-regular such that $(\alpha(x_1),...,\alpha(x_n))$ is the maximal ideal of $\pi^*A$. As usual, $X$ will denote the sequence $x_1,...,x_n$.

\begin{lemma}
\label{lem: Kernel of mapping from a field is zero}
Suppose that $\alpha$ is surjective. Let $M$ be an $A$-module and let $f:A/X \to M$ be a morphism in $\Derived(R)$ such that $\pi^*f \neq 0$. Then the kernel of $\pi^*f$ is zero.
\end{lemma}
\begin{proof}
Clearly, the morphism $u:A \to A/X$ induces the obvious surjection $\pi^*A \to (\pi^*A)/(X) \cong \pi^*(A/X)$. Suppose $y$ is a non-zero element in the kernel of $\pi^nf$. There exists $y' \in \pi^nA$ such that $y=u y'$. Choose an element $y'' \in \ext_{R^e}^*(A,A)$ such that $\alpha(y'') = y'$. Consider the morphism $\varphi:\Sigma^n A/X \to A/X$ which is the action of $y''$. Clearly $\pi^*(\varphi)(1)=y$. Note that $\varphi$ is invertible: since $y \in \pi^*(A/X)$ is invertible we can choose an element $z'' \in \ext^*_{R^e}(A,A)$ such that $y^{-1}=u \alpha(z'')$, then the action of $z''$ on $A/X$ is the inverse of $\varphi$. Hence $\pi^*\varphi$ is an isomorphism and therefore for every $b \in \pi^*(A/X)$ we have that $b=ay$ for some $a \in \pi^* A$. But this implies $\pi^*f = 0$, in contradiction.
\end{proof}


\begin{lemma}
\label{lem: Field module is direct sum}
Suppose that $\alpha$ is surjective. Then for every $A$-module $M$ the object $M/X$ is equal to a direct sum of copies of suspensions of $A/X$. In particular, if $M/X\neq 0$ then $A/X$ is a retract of $M/X$.
\end{lemma}
\begin{proof}
From Corollary~\ref{cor: Homotopy of M/X is associative module} it follows that $\pi^*(M/X)$ is a vector space over the graded field $\pi^*(A/X)$. Choose a basis $B \subset \pi^*(M/X)$. By Lemma~\ref{lem: Ring object maps surject module homotopy} for every $b \in B$ there is a morphism $\Psi_b:\Sigma^{|b|} A/X \to M$ such $(\pi^*\Psi_b)(1)=b$. From Lemma~\ref{lem: Kernel of mapping from a field is zero} we see that $\Ker \pi^*\Psi_b =0$. Hence the morphism:
\[ \bigoplus_{b \in B} \Psi_b: \bigoplus_{b \in B}\Sigma^{|b|} A/X \to M\]
induces an isomorphism on homotopy groups and is therefore an isomorphism.
\end{proof}


\begin{corollary}
\label{cor: The localizing subactegory generated by a field like is minimal}
Suppose that $\alpha$ is surjective. Then the localizing subcategory of $\Derived(R)$ generated by $A/X$ is minimal.
\end{corollary}
\begin{proof}
Let $\Tri$ be the image of $L_\uu$. Then $\Tri$ is a triangulated category with a compact generator $A$ and an action of $\ext^*_{R^e}(A,A)$. Let $S$ be the subring of $\ext^*_{R^e}(A,A)$ generated by $X$, then $S$ is a Noetherian commutative ring. Now we can apply the machinery of stratification to the action of $S$ on $T$.

Let $M \in \loc_\Tri(A/X)$ be a non-zero object, hence $\Gamma_{\uu(X)}M \cong M$ and by \cite[Proposition 2.11]{BIKstratifying} this implies that $M/X$ is non-zero. By Lemma~\ref{lem: Field module is direct sum}, $M/X$ is isomorphic to a direct sum of copies of $A/X$. Hence $M$ builds $M/X$ which builds $A/X$, so $M$ generates $\Im\Gamma_{\uu(X)}$.
\end{proof}


\section{Proofs of the main results}
\label{sec: Proofs of the main results}

\subsection*{Proof of Theorem~\ref{the: Main algebraic theorem on stratification}}
Let $\pp$ be a prime ideal of $\Hring$. To prove the theorem we need only show that the localizing subcategory $\Im \Gamma_\pp$ is either zero or minimal. Let $\uu$ be the specialization closed set $\zz(\pp)$. If $L_\uu R = 0$ then we are done. Otherwise, let $A$ be $L_\uu^e R$. Thus, as a left $R$ module, $A\cong R_\pp$. Observe that, in this case, the kernel of the map $\rho:\ext^*_{R^e}(R,R) \to \pi^*R$ must be contained in $\pp$ and hence $\rho(\pp)$ is a prime ideal of $\pi^*R$. The upshot of this observation is that $\pi^*R_\pp = \pi^*R_{\rho(\pp)}$. Since $\pi^*R$ is a regular ring, its localization at a prime ideal is a regular local ring. We conclude that $\pi^*A = \pi^*R_\pp$ is a regular local ring.

Let $Z=(z_1,...,z_q)$ be a sequence of elements in $\Hring$ that generates $\pp$. By~\cite[Proposition 2.11]{BIKstratifying}, the localizing subcategory of $\Derived(R)$ generated by $A/Z$ is equal to the localizing subcategory generated by $\Gamma_{\vv(Z)}R_\pp\cong\Gamma_\pp R$. Since $R$ generates $\Derived(R)$, the object $\Gamma_{\vv(Z)}R_\pp$ generates $\Im \Gamma_\pp$, and so does $A/Z$. Hence we must show that the localizing subcategory generated by $A/Z$ is minimal.

We turn to work inside the localizing subcategory $\Im L_\uu$, which we denote by $\Tri$. As noted before, this is a triangulated subcategory with a compact generator $A$ and an action of $\Hring_\pp$. It is important to remember that this action splits through the action of $\ext^*_{R^e}(A,A)$ on $\Tri$. By Corollary~\ref{cor: Hochschild cohomology of localization}, $\ext^*_{R^e}(A,A) \cong \ext^*_{R^e}(R,R)_\pp$ and the map $\alpha:\ext^*_{R^e}(A,A) \to \pi^*A$ is surjective. It is easy to see that $\Hring_\pp$ is a subring of $\ext^*_{R^e}(A,A)$ and that $\alpha$ restricted to $\Hring_\pp$ is also a surjection.

Let $X=(x_1,...,x_n)$ be a sequence of elements in $\Hring_\pp$ such that $(\alpha(x_1),...,\alpha(x_n))$ is a regular sequence in $\pi^*A$ which generates the maximal ideal of $\pi^*A$. We come now to a delicate point. Since we did not assume that $\Hring$ is isomorphic to $\pi^*R$, the sequence $X$ need not generate the maximal ideal of $\Hring_\pp$. Let $Y=(y_1,...,y_m)$ be a sequence of elements such that $X \cup Y$ generates the maximal ideal of $\Hring_\pp$. Since $X \cup Y$ and $Z$ generate the same ideal in $\Hring_\pp$, then by~\cite[Proposition 2.11]{BIKstratifying} the localizing subcategory generated by $A/X/Y$ is equal to the localizing subcategory generated by $A/Z$.

Our next step is to show that $A/X$ and $A/X/Y$ generate the same localizing subcategory. We do this by making a particular choice of $Y$, such that $Y$ is contained in the kernel of the map $h:\Hring_\pp \to \pi^*A$. Suppose, by induction, that we chose $y_1,..,y_t \in \ker h$. Let $\mm$ be the maximal ideal of $\Hring_\pp$ and let $y$ be a non-zero element in $\mm \setminus (x_1,..,x_n,y_1,...,y_t)$. Then there exists and element $x \in (x_1,..,x_n)$ such that $h(y-x) = 0$. Set $y_{t+1}$ to be $y-x$. Then the ideal generated by $x_1,...,x_n,y_1,...,y_{t+1}$ is equal to the ideal generated by $x_1,..,x_n,y_1,...,y_t,y$. The Noetherian property of $\Hring_\pp$ ensures this process will end with a set of generators for $\mm$.

It follows that for every $y \in Y$, the action of $y$ on $\pi^*A/X$ is zero. Therefore $A/X$ is $\vv(Y)$-torsion and $\Gamma_{\vv(Y)}(A/X) \cong A/X$. From~\cite[Proposition 2.11]{BIKstratifying} we see that $\Gamma_{\vv(Y)}(A/X)$ and $A/X/Y$ generate the same localizing subcategory.

It remains to show that the localizing subcategory generated by $A/X$ is minimal. But this is simply invoking Corollary~\ref{cor: The localizing subactegory generated by a field like is minimal}.\qed

\subsection*{Proof of Theorem~\ref{thm: Answer theorem}}
Let $R=\chains_*(M;Hk)$, where $Hk$ is the Eilenberg-Mac Lane spectrum of $k$. Suppose that Theorem~\ref{the: Main algebraic theorem on stratification} can be applied to this case. Clearly, we can choose $\Hring$ to be such that the composition $\Hring \to HH^*(R) \to \pi_{-*}R = H_{-*}(\Omega M;k)$ is an isomorphism. Thus $\Hring \cong k[x_1,...,x_n]$ where $x_i$ are in positive even degrees.

Since $\Hring$ stratifies $\Derived(R)$ and $\pi^*R$ is finitely generated over $\Hring$ there is an inclusion preserving bijection (see Theorem~\ref{the: BIK theorem on thick subcats}):
\[ \left\{
     \begin{array}{c}
        \text{Thick}\\
        \text{subcategories of }\Derived(R)^\compact\\
     \end{array}
   \right\}
   \xrightarrow{\supp_\Hring(-)}
   \left\{
     \begin{array}{c}
        \text{Specialization closed}\\
        \text{subsets  of }\supp_\Hring \Derived\\
     \end{array}
   \right\}
   \]
Thus, for every compact $R$-module $X$, if $\supp_\Hring \thick_{\Derived(R)}Hk \subseteq \supp_\Hring \thick_{\Derived(R)}X$ then $HK$ is finitely built by $X$. Using Lemma~\ref{lem: familiar support} we can translate this to the following statement: if $\supp_\Hring k \subseteq \supp_\Hring \pi^*X$ then $k$ is finitely built by $X$. In addition, our choice of $\Hring$ implies that the support of $\pi^*X$ as an $\Hring$-module is the same as the support of $\pi^*X$ as a $\pi^*R$-module. The support of $k$ is easy to compute: 
\[\supp_\Hring k = \{ \pp \in \spec \Hring \ | \ (x_1,...,x_n) \subseteq \pp\} \cong \spec(k)\]

Let $N \subset M$ be a connected submanifold of $M$, then $H_0(\fib_N;k)$ is a coproduct of copies of $k$. Let $X$ be the compact $R$-module $\chains_*(\fib_N;Hk)$, hence $\pi_0 X$ is a coproduct of copies of $k$. It is now easy to see that $(\pi^*X)_\pp \neq 0$ for every $\pp \in \supp_\Hring k$. Therefore $X$ finitely builds $k$. By Proposition~\ref{pro: when varphi is an iso} the map
\[ \varphi_N: H_{*+m}(\fl M;k) \to HH^*\Hom_{\St_M}(N,N)\]
is an isomorphism.

Thus, we must show that we can apply Theorem~\ref{the: Main algebraic theorem on stratification} to this case. There is only one condition that needs verifying: that $R$ is compact as an $R^e$-module. Consider the diagonal map $\delta:\Omega M \to \Omega M \times \Omega M^\op$, which is a map of topological monoids. It is well known that the Borel construction $((\Omega M \times \Omega M^\op) \times E\Omega M)/ \Omega M$ is equivalent to $\Omega M$ as an $\Omega M \times \Omega M^\op$-space. Applying the functor $\chains_*(-;Hk)$ to these constructions we see that there is a map of $Hk$-algebras $\delta:R \to R^e$ such that $R^e \otimes_{\delta R} Hk \simeq R$ (we write $\delta R$ to emphasize where the $R$-module structure comes from).

Thus, if $k$ is compact as an $R$-module then $R$ is compact as an $R^e$-module. By~\cite[Proposition 5.3]{DwyerGreenleesIyengar}, since $M$ is a compact manifold then $Hk$ is indeed a compact $R$-module.
\qed

\subsection*{Proof of Theorem~\ref{thm: Second theorem}}
Let $R=\chains_*(M;Hk)$, where $Hk$ is the Eilenberg-Mac Lane spectrum of $k$. From the proof of Theorem~\ref{thm: Answer theorem} above we see that for any compact $R$-modules $X$ and $Y$:
\[ Y \in \thick_{\Derived(R)} X \quad \Leftrightarrow \quad \supp_{\pi^*R} \pi^*Y \subseteq \supp_{\pi^*R} \pi^*X\]

Now let $N_1$ and $N_2$ be two connected, orientable, submanifolds of $M$. Set $X=\chains_*(\fib_{N_1};Hk)$ and $Y = \chains_*(\fib_{N_2};Hk)$. Then $Y$ is built by $X$ if and only if
\[ \supp_{H_*(\Omega M;k)} H_*(\fib_{N_2};k) \subseteq \supp_{H_*(\Omega M;k)} H_*(\fib_{N_1};k)\]

Note that $Y$ is built by $X$ if and only if the $X$-colocalization of $Y$ is equivalent to $Y$. By the Dwyer-Greenlees formula for colocalization (Theorem~\ref{thm: Dwyer-Greenlees colocalization and completion formulas}) the $X$-colocalization of $Y$ is the morphism
\[ \Hom_R(X,Y) \otimes_{\End_R(X)} X \to Y\]

Using the definition of the string topology category we conclude that the colocalization morphism
\[\gamma:\Hom_{\St_M}(N_1,N_2) \otimes_{\End_{\St_M}(N_1)} \chains_*(\fib_{N_1};Hk) \to \chains_*(\fib{N_2};Hk)\]
is an equivalence if and only if
\[ \supp_{H_*(\Omega M;k)} H_*(\fib_{N_2};k) \subseteq \supp_{H_*(\Omega M;k)} H_*(\fib_{N_1};k)\]

To complete the proof we need to take a short detour. Let $\ee$ be the function spectrum $F(\Sigma^\infty M;Hk)$, note that $\ee$ is a commutative $Hk$-algebra. Since $M$ is simply connected then $\End_\ee(Hk) \simeq R$ as $Hk$-algebras~\cite{DwyerGreenleesIyengar}. The results of Dwyer and Greenlees~\cite{DwyerGreenlees} show there are adjoint functors
\[ Hk \otimes_R :\Derived(R) \rightleftarrows \Derived(\ee) : \Hom_\ee(Hk,-)\]
which yield an equivalence of categories when restricted to $\Derived(R)^\compact$ on the left hand side and $\thick_{\Derived(\ee)}Hk$ on the right hand side. In addition, it is well known that $Hk \otimes_R \chains_*(\fib_{N_i};Hk) \simeq \chains_*(N_i;Hk)$.

We conclude that $\gamma$ is an equivalence if and only if $Hk \otimes_R \gamma$ is an equivalence. Note that $\chains_*(\fib_{N_1};k)$ is an $R\otimes \End_{\St_M}(N_1)$-module. The $R$-module structure on $\Hom_{\St_M}(N_1,N_2) \otimes_{\End_{\St_M}(N_1)} \chains_*(\fib_{N_1};k)$ comes from the left $R$-module structure on $\chains_*(\fib_{N_1};k)$. Thus:
\begin{align*}
Hk \otimes_R & (\Hom_{\St_M}(N_1,N_2) \otimes_{\End_{\St_M}(N_1)} \chains_*(\fib_{N_1};Hk)) \\
&\simeq \Hom_{\St_M}(N_1,N_2) \otimes_{\End_{\St_M}(N_1)} (Hk \otimes_R \chains_*(\fib_{N_1};Hk)) \\
&\simeq \Hom_{\St_M}(N_1,N_2) \otimes_{\End_{\St_M}(N_1)} \chains_*(N_1;Hk)
\end{align*}
Hence $Hk \otimes_R \gamma$ is the morphism: $\Hom_{\St_M}(N_1,N_2) \otimes_{\End_{\St_M}(N_1)} \chains_*(N_1;Hk) \to \chains_*(N_2;Hk)$.
\qed

\subsection*{The examples}
We now show two cases in which the main theorems of this paper hold.
\begin{theorem}
\label{the: The two exmaples of compact Lie groups}
Theorems~\ref{thm: Answer theorem} and~\ref{thm: Second theorem} hold in the following two cases:
\begin{enumerate}
\item $k=H\mathbb{Q}$ and $M$ is a simply-connected, compact, Lie group;
\item $k=H\Int$ and $M=SU(n)$ for $n>1$.
\end{enumerate}
\end{theorem}

The proof is based on the following simple lemma.

\begin{lemma}
\label{lem: conditions for surjection of HH}
Suppose that $k$ is either a field or the ring of integers $\Int$. Let $R$ be an $Hk$ algebra such that
\begin{enumerate}
\item $\pi^* R = k[x_1,...,x_n]$ and
\item $HH^*R = k[x_1,...,x_n] \otimes \Lambda_k[y_1,...,y_n]$ with $\deg y_i = -\deg x_i - 1$.
\end{enumerate}
Then the natural map $h:HH^*R \to \pi^*R$ is surjective.
\end{lemma}
\begin{proof}
We will use the spectral sequence of~\cite[Theorem IV.4.1]{EKMM}. Hence there is a conditionally convergent spectral sequence
\[ E^{p,q}_2 = \ext^*_{\pi^*(R^e)}(\pi^*R,\pi^*R) \ \Rightarrow \ HH^*R\]
Our assumptions on $\pi^*R$ imply that $\pi^*(R^e) \cong (\pi^*R)^e$. Therefore, the $E_2$-term of the spectral sequence is in fact $HH^{*,*}(\pi^*R)$.

We next show that the spectral sequence collapses on the $E_2$-term. A simple calculation shows that $E_2 = k[\tilde{x}_1,...,\tilde{x}_n]\otimes \Lambda_k[\tilde{y_1},...,\tilde{y_n}]$ where the bi-degree of $x_i$ is $(0,\deg x_i)$ and the bi-degree of $y_i$ is $(-1,-deg x_i)$ (we use here homological bi-degrees). Since the spectral sequence has finitely many non-zero columns, there is strong convergence. First, suppose that $k$ is a field. We see that for every $n$, the vector space $\oplus_{p+q=n} E_2^{p,q}$ is isomorphic to $HH^n R$, i.e. it has the correct dimension. A non-trivial differential on the $E_2$-term implies that for some $n$, the dimension of $\oplus_{p+q=n} E_\infty^{p,q}$ is smaller than that of $\oplus_{p+q=n} E_2^{p,q}$, hence there can be no non-trivial differential. If $k=\Int$ we use a similar dimension argument on $\mathbb{Q} \otimes (\oplus_{p+q=n} E_2^{p,q})$. Thus, the spectral sequence collapses on the $E_2$-term.

To complete the proof it suffices to show that the image of $h$ is the column $E_\infty^{0,*}$. This is a fairly standard argument, but for completeness we will bring it here. There are two ways to define the map $h$. First, the action of $HH^*(R)$ on $\Derived(R)$ means that for every $z\in HH^*(R)$ we have an element $z_R:\ext^*_{\Derived(R)}(R,R) = \pi^*R$. This is the definition we used throughout this paper. Second, the morphism $R^e \to R$ in $\Derived(R^e)$ induces a map
\[ \ext^*_{R^e}(R,R) \to \ext^*_{R^e}(R^e,R) \cong \pi^*R\]
It is easy to verify that both definitions agree. Here we will use the second definition.

The morphism $R^e \to R$ also induces a map of spectral sequences $E_*h$: from the spectral sequence whose $E_2$ term is $\ext^*_{\pi^*(R^e)}(\pi^*R,\pi^*R)$ to the one whose $\bar{E}_2$-term is $\ext^*_{\pi^*(R^e)}(\pi^*(R^e),\pi^*R)$. This map converges to the map $h$. The $E_2$-term of the second spectral sequence is non-zero only for $p=0$, where it is simply $\bar{E}_2^{0,*} = \pi^*R$. Thus, both spectral sequences collapse on the $E_2$-term. The map induced on the zero'th columns of both spectral sequences is therefore the map $h$ (up to extensions):
\[ \ext^0_{\pi^*(R^e)}(\pi^*R,\pi^*R) \to \ext^0_{\pi^*(R^e)}(\pi^*(R^e),\pi^*R) = \pi^*R\]
The left hand side of this map is simply the center of $\pi^*R$. In fact, this map is simply the injection of the center of $\pi^*R$ into $\pi^*R$. Since $\pi^*R$ is commutative, the map is surjective.
\end{proof}

\begin{proof}[Proof of Theorem~\ref{the: The two exmaples of compact Lie groups}]
Hepworth~\cite{Hepworth} calculated the homology of the free loop space $\fl G$ where $G$ is a compact Lie group of dimension $m$. His results show that $H_{*+m}(\fl G;\Int) = H_*(\Omega G;\Int) \otimes H_{*+m}(G;\Int)$ as rings (the product on $H_{*+m}(G;\Int)$ is the one induced from the cup product on cohomology using Poincar\'e duality). Using Malm's isomorphism~\cite{Malm} $HH^{-*}\chains_*(\Omega G;\Int) \cong H_{*+m}(\fl G;\Int)$ this yields a calculation of the Hochschild cohomology of $\chains_*(\Omega G;\Int)$.

Suppose $G$ is simply connected and $k=\mathbb{Q}$. Then $H_*(\Omega G;k)$ is a polynomial ring and $H^*(G;k)$ is an exterior algebra. It is easy to see that in this case $HH^*\chains_*(\Omega G;k)$ satisfies the conditions of Lemma~\ref{lem: conditions for surjection of HH}. Thus, Theorems~\ref{thm: Answer theorem} and~\ref{thm: Second theorem} hold in this case.

When $G=SU(n)$ for $n>1$ and $k=\Int$, again $H_*(\Omega G;k)$ is a polynomial ring, $H^*(G;k)$ is an exterior algebra and
\[ HH^*\chains_*(\Omega G;k) \cong H_*(\Omega G;k) \otimes H^{-*}(G;k)\]
Lemma~\ref{lem: conditions for surjection of HH} applies in this case as well.
\end{proof}

\bibliographystyle{plain}       
\bibliography{bib2010}          

\end{document}